\documentclass[11pt,3p,times]{elsarticle}

\usepackage{srcltx}
\usepackage{xcolor}
\usepackage{colortbl}
\usepackage{enumitem}
\usepackage{amssymb, color}
\usepackage{hyperref}
\usepackage{amsmath}
\usepackage{pifont}
\usepackage[utf8]{inputenc}
\makeatletter
\def\LaTeX{\leavevmode L\raise.42ex
    \hbox{\kern-.3em\size{\sf@size}{0pt}\selectfont A}\kern-.15em\TeX}
\makeatother

\newcommand{\BibTeX}{{\rm B\kern-.05em{\sc
          i\kern-.025emb}\kern-.08em\TeX}}
\makeatletter
\def\@currentlabel{2.1}\label{e:dispaa}
\def\@currentlabel{2.21}\label{e:dispau}
\def\@currentlabel{2.22}\label{e:dispav}
\def\@currentlabel{2.23}\label{e:dispaw}
\def\@currentlabel{2.24}\label{e:dispax}
\def\theequation{\thesection.\@arabic\c@equation}
\makeatother

\renewcommand{\theequation}{\arabic{section}.\arabic{equation}}

\everymath{\displaystyle}
\newtheorem{theorem}{Theorem}[section]
\newtheorem{thm}{Theorem} [section]
\newtheorem{lem}{Lemma} [section]
\newtheorem{prop}{Proposition} [section]

\newtheorem{proof}{Proof}[section]

\renewcommand{\theequation}{\thesection.\arabic{equation}}
\renewcommand{\thesection}{\arabic{section}}
\renewcommand{\theequation}{\thesection.\arabic{equation}}

\let\ssection=\section\renewcommand{\section}{\setcounter{equation}{0}\ssection}
\begin{document}
\begin{frontmatter}
\title{A degenerate Kirchhoff-type problem  involving variable $s(\cdot)$-order fractional $p(\cdot)$-Laplacian  with weights}

\author[OMA,ALL]{Mostafa Allaoui \corref{cor1}}
\ead{m.allaoui@uae.ac.ma}
\author[mk0,mk1,mk2]{Mohamed Karim Hamdani}
\ead{hamdanikarim42@gmail.com}
\author[L]{Lamine Mbarki}
\ead{mbarki.lamine2016@gmail.com}

\begin{center}
\address[OMA]{Department of Mathematics, FSTH, Abdelmalek Essaadi University, Tetuan, Morocco.}
\address[ALL]{Department of Mathematics, FSO,  Mohammed I University,  Oujda, Morocco.}
\address[mk0]{Science and technology for defense lab LR19DN01, center for military research, military academy, Tunis, Tunisia.}
\address[mk1]{Military Aeronautical Specialities School, Sfax, Tunisia.}
\address[mk2]{Department of Mathematics, University of Sfax, Faculty of Science of Sfax, Sfax, Tunisia.}
\address[L]{Mathematics Departement, Faculty of Science of Tunis , University of Tunis El Manar, Tunisia.}

\end{center}

\begin{abstract}
This paper deals with a class of nonlocal variable $s(.)$-order fractional $p(.)$-Kirchhoff type equations:
\begin{eqnarray*}
\left\{
 \begin{array}{ll}
 \mathcal{K}\left(\int_{\mathbb{R}^{2N}}\frac{1}{p(x,y)}\frac{|\varphi(x)-\varphi(y)|^{p(x,y)}}{|x-y|^{N+s(x,y){p(x,y)}}} \,dx\,dy\right)(-\Delta)^{s(\cdot)}_{p(\cdot)}\varphi(x)
=f(x,\varphi)
 \quad \mbox{in }\Omega, \\
 \\
\varphi=0     \quad \mbox{on }\mathbb{R}^N\backslash\Omega.
\end{array}
\right.
\end{eqnarray*}
Under  some suitable conditions on the functions $p,s, \mathcal{K}$ and $f$, the existence and multiplicity of nontrivial  solutions for the above problem are obtained. Our results cover the degenerate case in the $p(\cdot)$ fractional setting.
\end{abstract}

\end{frontmatter}

\section{Introduction}
 In this work, we investigate the following variable $s(.)$-order fractional $p(.)$-Kirchhoff type problem:
  \begin{eqnarray}\label{main}
\left\{
 \begin{array}{ll}
 \mathcal{K}\left(\int_{\mathbb{R}^{2N}}\frac{1}{p(x,y)}\frac{|\varphi(x)-\varphi(y)|^{p(x,y)}}{|x-y|^{N+s(x,y){p(x,y)}}} \,dx\,dy\right)(-\Delta)^{s(\cdot)}_{p(\cdot)}\varphi(x)
=f(x,\varphi)
 \quad \mbox{in }\Omega, \\
 \\
\varphi=0     \quad \mbox{on }\mathbb{R}^N\backslash\Omega,
\end{array}
\right.
\end{eqnarray}
where $\mathcal{K}$ is a model of Kirchhoff coefficient, $ \Omega$ is a smooth bounded domain in $\mathbb{R}^{N}$ and $f:\Omega\times\mathbb{R}\to \mathbb{R}$ is a continuous function specified later. The operator $(-\Delta)^{s(\cdot)}_{p(\cdot)}$ is called variable
$s(.)$-order fractional $p(.)$-Laplacian, given $s(.): \overline{\Omega}\times \overline{\Omega}\rightarrow(0,1)$ and  $p(.): \overline{\Omega}\times \overline{\Omega}\rightarrow(1,+\infty)$ with $N > s(x,y)p(x, y)$ for all $(x, y)\in \overline{\Omega}\times \overline{\Omega}$, which can be
defined as
\begin{equation*}
  (-\Delta)^{s(\cdot)}_{p(\cdot)}\varphi(x):=P.V \int_{\mathbb{R}^N} \frac{|\varphi(x)-\varphi(y)|^{p(x,y)-2}(\varphi(x)-\varphi(y))}{|x-y|^{N+s(x,y){p(x,y)}}} \,dy, ~~x\in \mathbb{R}^N,
  \end{equation*}
  along any  $\varphi\in C_0^\infty (\mathbb{R}^N)$, where P.V. denotes the Cauchy principle value.\\

Throughout this paper, we make the following assumptions:
\begin{description}
   \item[(A1)] $\mathcal{K}: \mathbb{R}_0^+\to \mathbb{R}_0^+$ is a continuous function and satisfies the (polynomial growth) condition
\begin{equation*}
  k_1\zeta^{\theta-1}\leq \mathcal{K}(\zeta)\leq k_2\zeta^{\theta-1} ~~\text{for any} ~~ \zeta\geq0,
\end{equation*}
where $0<k_1\leq k_2$ are real numbers and exponent $1<\theta$;
\end{description}
and the variable exponent $p$ and the variable order $s$ admit the following statements:
\begin{itemize}
\item[\rm(S)] $s(x,y)$ is symmetric function, ie, $s(x,y)=s(y,x)$ and we have
$$0<s^{-}:=\inf_{(x,y)\in \overline{\Omega}\times \overline{\Omega}}s(x,y)\leq s^{+}:=\sup_{(x,y)\in \overline{\Omega}\times \overline{\Omega}}s(x,y)<1.$$
\item[\rm(P)] $p(x,y)$ is symmetric function, ie, $p(x,y)=p(y,x)$ and we have
$$1<p^{-}:=\inf_{(x,y)\in \overline{\Omega}\times \overline{\Omega}}p(x,y)\leq p^{+}:=\sup_{(x,y)\in \overline{\Omega}\times \overline{\Omega}}p(x,y)<\infty.$$
\end{itemize}
 For any $x\in\mathbb R^N$, we denote
$$
\overline{p}(x):=p(x,x), \quad \overline{s}(x):=s(x,x).
$$

 In recent years, a great deal of attention has been paid to the study of problems
involving fractional $p, p(.)$--Laplacian, both in the pure mathematical research
and in the concrete real-world applications, such as, optimization, finance,
continuum mechanics, phase transition phenomena, population dynamics,
and game theory, see \cite{AHKC,ASSA,BR,BFL,Ch,FP,HZCR,MT,ZAF,ZYL} and references therein.\\

In \cite{AAA}, using variational methods, the authors studied a nonlocal $p(x)$--Kirchhoff problem:
\begin{eqnarray*}
\left\{
 \begin{array}{ll}
 M\left(\int_{\Omega}\frac{|\nabla u|^{p(x)}}{p(x)} \,dx\right)(-\Delta_{p(x)}u)
=f(x,u)
 \quad \mbox{in }\Omega, \\
 \\
u=0     \quad \mbox{on }\partial\Omega,
\end{array}
\right.
\end{eqnarray*}
and have obtained the existence and multiplicity of the above problem, under appropriate assumptions on $f$ and $M$.\\

In \cite{ABS}, by direct variational approach and Ekeland's variational principle, Azroul et al investigate the existence of nontrivial weak solutions for the following problem:
\begin{eqnarray*}
\left\{
 \begin{array}{ll}
 M\left(\int_{\mathcal{Q}}\frac{1}{p(x,y)}\frac{|u(x)-u(y)|^{p(x,y)}}{|x-y|^{N+s{p(x,y)}}} \,dx\,dy\right)(-\Delta_{p(x)})^{s}u(x)
=\lambda|u(x)|^{r(x)-2}u(x)
 \quad \mbox{in }\Omega, \\
 \\
u=0     \quad \mbox{on }\mathbb{R}^N\backslash\Omega.
\end{array}
\right.
\end{eqnarray*}

In \cite{XZY}, the authors considered a multiplicity result for a Schr\"{o}dinger equation driven by the variable $s(.)$-order fractional Laplace operator via variational methods. However, the authors in \cite{WZ}, investigated the existence of infinitely many solutions for a kind of Kirchhoff-type variable $s(.)$-order problem by using four different critical point theorems.\\

Inspired by the above results, we study problem \eqref{main} in the case when the function $\mathcal{K}$ is singular at zero and $f$ involves indefinite weight functions.

\section{Functional setting and preliminaries}
In this part,  we briefly review some basic knowledge,  properties of Lebesgue spaces with
variable exponents. For more details, the reader can refer to \cite{Die, Fan1} and the references therein.
For this aim, set
$$
C_+(\Omega):=\{g:g\in C(\overline{\Omega})\ \text{and}\ g(x)>1,
\forall x\in\overline{\Omega}\}.
$$
For $g(\cdot)\in C_+(\Omega)$,  the variable exponent Lebesgue
space $L^{g(\cdot)}(\Omega)$ is defined by
$$
L^{g(\cdot)}(\Omega):=\{w:\Omega\to\mathbb{R}\text{ measurable and }
 \int_{\Omega}|w(x)|^{g(x)}dx<\infty\}.
$$
This space is endowed with the so-called Luxemburg norm given by
$$
\| w \|_{L^{g(x)}(\Omega)}=|w|_{g(\cdot)}:=\inf\{\delta>0:\int_{\Omega}|\frac{w(x)}{\delta}|^{g(x)}dx\leq1\}
$$
and $( L^{g(\cdot)}(\Omega),|w|_{g(\cdot)})$ becomes a Banach space, and we
call it variable exponent Lebesgue space.
\begin{prop}\cite{Fan1} \label{prop2.1} \quad
For every  $\varphi\in L^{g(x)}(\Omega)$ and $\psi\in L^{r(x)}(\Omega)$, we have
\[
\left|\int_{\Omega}\varphi\psi\,dx\right|\leq
\left(\frac{1}{g^{-}}+\frac{1}{r^{-}}\right)|\varphi|_{g(x)}|\psi|_{r(x)},
\]
 where $r,g\in C_+(\Omega)$ and $1/r(x)+1/g(x)=1$.\\
Moreover, if $g_1, g_2, g_3\in C_+(\overline{\Omega})$ and $\frac{1}{g_1(x)}+\frac{1}{g_2(x)}+\frac{1}{g_3(x)}=1$, then for any $\varphi\in L^{g_1(x)}(\Omega)$, $\chi\in L^{g_2(x)}(\Omega)$ and $\psi\in L^{g_3(x)}(\Omega)$ the following inequality holds:
 \begin{equation}\label{eh}
    \int_{\Omega}|\varphi\chi\psi|dx\leq \left(\frac{1}{g_1^-}+\frac{1}{g_2^-}+\frac{1}{g_3^-}\right)| \varphi |_{g_1(x)}| \chi |_{g_2(x)}| \psi|_{g_3(x)}.
 \end{equation}
\end{prop}
\begin{prop} \cite{Cekic} \label{lem23}
 Assume that $h\in L^\infty_+(\Omega)$, $g\in C_+(\overline{\Omega})$. If $|\chi|^{h(x)}\in L^{g(x)}(\Omega)$, then we have
 \begin{equation}\label{eh1}
      \min\left\{|\chi|_{h(x)g(x)}^{h^-},|\chi|_{h(x)g(x)}^{h^+}\right\}\leq \left||\chi|^{h(x)}\right|_{g(x)}\leq \max\left\{|\chi|_{h(x)g(x)}^{h^-},|\chi|_{h(x)g(x)}^{h^+}\right\}.
 \end{equation}
  \end{prop}

In the present part, we give the variational setting of problem (1.1) and state important results to be used later.   We set $\mathcal{Q}:=\mathbb R^{2N}\setminus(C^\Omega_{R^{N}}\times C^\Omega_{R^{N}})$ and define the
 fractional Sobolev space with variable exponent as
\begin{equation*}
E :=\left\{v:\mathbb R^N \to\mathbb R:v_{|_\Omega}\in L^{\overline{p}(x)}(\Omega),
 \quad \int_{\mathcal{Q}}\frac{|v(x)-v(y)|^{p(x,y)}}{\eta^{p(x,y)}| x-y|^{N+s(x,y)p(x,y)}}
 \,dx\,dy<\infty,\text{ for some } \eta>0\right\}.
 \end{equation*}
The space ${E}$ is equipped with the  norm
$$\| v\|_{ {E}}:=\| v\|_{L^{\overline{p}(x)}(\Omega)}
 +[v]_{E},$$
where $[v]_{E}$ is the seminorm defined as follows
$$
[v]_{E}=\inf\left\{\eta>0:\int_{\mathcal{Q}}\frac{|
 v(x)-v(y)|^{p(x,y)}}{\eta^{p(x,y)}| x-y |^{N+s(x,y)p(x,y)}}\,dx\,dy<1\right\}.
$$
Then $(E,\|\cdot\|_{E})$ is a separable reflexive Banach space.\\
 Now, let define the subspace ${E}_0$ of ${ E}$ as
$$
 E_0  :=\left\{v\in { E}: v=0\text{ a.e.\ in }C^\Omega_{R^{N}}\right\},
$$
with the norm on $E_0$
\[
 \| v\|_{{ E}_0}=[v]_{E}.
\]
\begin{prop}\label{prp 3.3}\cite{BT}
Let $s(\cdot)$ and $p(\cdot)$  satisfy {\rm (S)} and {\rm (P)}  with $s(x,y)p(x,y)<N$ for any $(x,y)\in \overline{\Omega}\times\overline{\Omega}$.  Then for any $g\in C_+(\overline{\Omega})$  such that $1<g^-\leq g(x)< p_s^*(x):=\frac{N\overline{p}(x)}{N-\overline{s}(x)\overline{p}(x)}$ for any
$x\in \overline{\Omega}$, there exits a positive constant $C=C(N,s,p,g,\Omega)$ such that
\begin{equation*}
  \| w \|_{L^{g(x)}(\Omega)}\leq C \| w \|_{{ E}_0},
\end{equation*}
for every $w\in{ E}_0$. Moreover, the embedding $E_0\hookrightarrow L^{g(x)}(\Omega)$is compact.
\end{prop}
Let us set the fractional modular function
   $\rho_{s,p}:{ E}_0\to  \mathbb R$ as
 \begin{equation} \label{modular}
 \rho_{s,p}(w):=\int_{\mathcal{Q}}\frac{|
 w(x)-w(y)|^{p(x,y)}}{| x-y |^{N+s(x,y)p(x,y)}}\,dx\,dy.
\end{equation}
\begin{prop}\label{lem 3.1}\cite{BT}
 Let $w,w_m \in { E}_0$ and $\rho_{s,p}$ be defined as in \eqref{modular}.
 Then we have the following results:
 \begin{itemize}
 \item[(i)] $ \| w \|_{{ E}_0}<1$ $(=1;>1)$  if and only if $\rho_{s,p}(w)<1(=1;>1)$.

 \item[(ii)] If $\| w\|_{{ E}_0}>1$, then
 $ \| w\|_{{ E}_0} ^{p^{-}}\leq\rho_{s,p}(w)\leq\| w \|_{{ E}_0}^{p^{+}}$.

 \item[(iii)] If $\| w \|_{{ E}_0}<1$, then
 $\| w\|_{{ E}_0} ^{p^{+}}\leq\rho_{s,p}(w)\leq\| w\|_{{ E}_0}^{p^{-}}$.
  \item[(iv)] ${\lim_{m\to  \infty} }\| w_m - w \|_{{ E}_0} =0 \Leftrightarrow {\lim_{m\to  \infty}} \rho_{s,p}(w_m -w)=0$.
 \end{itemize}
\end{prop}	
\begin{prop} \label{lem 3.3}\cite{BT}
 $(E_0,\|\cdot\|_{E_0})$ is a separable, reflexive and uniformly convex Banach space.
\end{prop}
\begin{prop}\label{propS}
 For all $u, \varphi \in E_0$, we consider the operator $\mathcal{T}: E_0\to E_0^*$
 such that
 \begin{equation*}
   \langle \mathcal{T}(u),\varphi\rangle=\int_{\mathcal{Q}}\frac{|u(x)-u(y)|^{p(x,y)-2}(u(x)-u(y))(\varphi(x)-\varphi(y))}{|x-y|^{N+sp(x,y)}}dxdy.
 \end{equation*}
 Then, the following assertions hold:
 \begin{enumerate}
   \item [(i)] $ \mathcal{T}$ is a bounded and strictly monotone operator;
   \item [(ii)] $ \mathcal{T}$ is a mapping of type $(S^+)$, that is,
   \begin{equation*}
     \text{if}~u_n\rightharpoonup u ~\text{in}~ E_0 ~\text{and}~\limsup  \langle \mathcal{T}(u_n)-\mathcal{T}(u),u_n-u\rangle\leq 0, ~\text{then}~ u_n\to u ~\text{in}~ E_0.
   \end{equation*}
    \end{enumerate}
\end{prop}
\begin{proof}
\begin{itemize}
  \item [(i)] Obviously, $\mathcal{T}$ is a bounded operator. Using  Simon's inequalities:
  \begin{equation}\label{Simon1}
    \left(|\xi|^{p-2}\xi -|\eta|^{p-2}\eta\right)\left(\xi-\eta\right)\geq \frac{1}{2^p}|\xi-\eta|^p~~\text{if}~p\geq2,
  \end{equation}
  \begin{equation}\label{Simon2}
    \left(|\xi|^{p-2}\xi -|\eta|^{p-2}\eta\right)\left(\xi-\eta\right)\left(|\xi|+|\eta|\right)^{2-p}\geq (p-1)|\xi-\eta|^2~~\text{if}~1<p<2,
  \end{equation}
for any $\xi,\eta\in \mathbb{R}^N$, we deduce that
  \begin{equation*}
    \langle \mathcal{T}(v)-\mathcal{T}(w),v-w\rangle>0~~\text{for}~v\neq w.
  \end{equation*}
Thus, $\mathcal{T}$ is strictly monotone.
  \item [(ii)]
Let $(u_n)$ be a sequence of $E_0$ such that
  \begin{equation*}
    u_n\rightharpoonup u ~\text{in}~ E_0 ~\text{and}~\limsup  \langle \mathcal{T}(u_n)-\mathcal{T}(u),u_n-u\rangle\leq 0.
  \end{equation*}
Then, from (i), we deduce that
  \begin{equation}\label{s+1}
    \lim\limits_{n\to +\infty}  \langle \mathcal{T}(u_n)-\mathcal{T}(u),u_n-u\rangle=0.
  \end{equation}
\end{itemize}
Put
$$
\mathcal{U}_p=\{(x,y)\in \mathcal{Q}:1<p(x,y)<2\},~~
\mathcal{V}_p=\{(x,y)\in \mathcal{Q}:p(x,y)\geq2\}.
$$
Let ($x,y)\in \mathcal{U}_p$ and $w_n=u_n-u$. Using \eqref{Simon2},  H\"{o}lder's inequality and Propositions \ref{lem23}, \ref{lem 3.1}, we get
\begin{equation}\label{s+0}
\begin{split}
  \int_{\mathcal{U}_p} & \frac{|w_n(x)-w_n(y)|^{p(x,y)}}{|x-y|^{N+s(x,y)p(x,y)}}dxdy \\
   &\leq \frac{1}{p^--1}\int_{\mathcal{U}_p} \left(\left[ \frac{|u_n(x)-u_n(y)|^{p(x,y)-2}(u_n(x)-u_n(y))(w_n(x)-w_n(y))}{|x-y|^{N+s(x,y)p(x,y)}}\right.\right.\\
   &\left.-\frac{|u(x)-u(y)|^{p(x,y)-2}(u(x)-u(y))(w_n(x)-w_n(y))}{|x-y|^{N+s(x,y)p(x,y)}}\right]^{\frac{p(x,y)}{2}}\\
   &\left.\times \left[\frac{|u_n(x)-u_n(y)|^{p(x,y)}+|u(x)-u(y)|^{p(x,y)}}{|x-y|^{N+s(x,y)p(x,y)}}\right]^{\frac{2-p(x,y)}{2}}\right)dxdy\\
   &\leq \frac{1}{p^--1}\int_{\mathcal{Q}} \left(\left[ \frac{|u_n(x)-u_n(y)|^{p(x,y)-2}(u_n(x)-u_n(y))(w_n(x)-w_n(y))}{|x-y|^{N+s(x,y)p(x,y)}}\right.\right.\\
   &\left.-\frac{|u(x)-u(y)|^{p(x,y)-2}(u(x)-u(y))(w_n(x)-w_n(y))}{|x-y|^{N+s(x,y)p(x,y)}}\right]^{\frac{p(x,y)}{2}}\\
   &\left.\times \left[\left(\frac{|u_n(x)-u_n(y)|^{p(x,y)}}{|x-y|^{N+s(x,y)p(x,y)}}\right)^{\frac{2-p(x,y)}{2}}+
   \left(\frac{|u(x)-u(y)|^{p(x,y)}}{|x-y|^{N+s(x,y)p(x,y)}}\right)^{\frac{2-p(x,y)}{2}}\right]\right)dxdy\\
   &= \frac{1}{p^--1}\int_{\mathcal{Q}}h_1^{\frac{p(x,y)}{2}}(x,y)\left(h_2^{\frac{2-p(x,y)}{2}}(x,y)+h_3^{\frac{2-p(x,y)}{2}}(x,y)\right)dxdy\\
   &\leq c \|h_1^{\frac{p(x,y)}{2}}\|_{L^{\frac{2}{p(x,y)}}(\mathcal{Q})}\left(\|h_2^{\frac{2-p(x,y)}{2}}\|_{L^{\frac{2}{2-p(x,y)}}(\mathcal{Q})}+
   \|h_3^{\frac{2-p(x,y)}{2}}\|_{L^{\frac{2}{2-p(x,y)}}(\mathcal{Q})}\right)\\
   &\leq c \left(\|h_1\|^{\frac{p^+}{2}}_{L^{1}(\mathcal{Q})}+\|h_1\|^{\frac{p^-}{2}}_{L^{1}(\mathcal{Q})}\right)
   \left(\|h_2\|^{\frac{2-p^+}{2}}_{L^{1}(\mathcal{Q})}+\|h_2\|^{\frac{2-p^+}{2}}_{L^{1}(\mathcal{Q})}+
   \|h_3\|^{\frac{2-p^+}{2}}_{L^{1}(\mathcal{Q})}+\|h_3\|^{\frac{2-p^-}{2}}_{L^{1}(\mathcal{Q})}\right),
\end{split}
\end{equation}
where
\begin{equation*}
  h_1(x,y)=\frac{|u_n(x)-u_n(y)|^{p(x,y)-2}(u_n(x)-u_n(y))(w_n(x)-w_n(y))}{|x-y|^{N+s(x,y)p(x,y)}}
   -\frac{|u(x)-u(y)|^{p(x,y)-2}(u(x)-u(y))(w_n(x)-w_n(y))}{|x-y|^{N+s(x,y)p(x,y)}},
\end{equation*}
\begin{equation*}
  h_2(x,y)=\frac{|u_n(x)-u_n(y)|^{p(x,y)}}{|x-y|^{N+s(x,y)p(x,y)}}~~\text{and}~ h_3(x,y)=\frac{|u(x)-u(y)|^{p(x,y)}}{|x-y|^{N+s(x,y)p(x,y)}}.
\end{equation*}
Then, by using \eqref{s+0} and Proposition \ref{lem 3.1}, we obtain
\begin{equation}\label{s+2}
\begin{split}
  \int_{\mathcal{U}_p} & \frac{|w_n(x)-w_n(y)|^{p(x,y)}}{|x-y|^{N+s(x,y)p(x,y)}}dxdy \\
   &\leq c \left(\langle \mathcal{T}(u_n)-\mathcal{T}(u),u_n-u\rangle^{\frac{p^+}{2}}+ \langle \mathcal{T}(u_n)-\mathcal{T}(u),u_n-u\rangle^{\frac{p^-}{2}}\right)\\
   &\times\left(\rho_{s,p}(u_n)^{\frac{2-p^+}{2}}+\rho_{s,p}(u_n)^{\frac{2-p^-}{2}}+\rho_{s,p}(u)^{\frac{2-p^+}{2}}+\rho_{s,p}(u)^{\frac{2-p^-}{2}}\right).
   \end{split}
\end{equation}
From \eqref{s+1} and \eqref{s+2}, we deduce that
\begin{equation}\label{s+3}
 \lim\limits_{n\to +\infty}\int_{\mathcal{U}_p}  \frac{|w_n(x)-w_n(y)|^{p(x,y)}}{|x-y|^{N+s(x,y)p(x,y)}}dxdy=0.
\end{equation}
For ($x,y)\in \mathcal{V}_p$. Using \eqref{Simon1},  H\"{o}lder's inequality and Propositions \ref{lem23}, \ref{lem 3.1}, we get
\begin{equation}\label{s+4}
  \int_{\mathcal{V}_p}\frac{|w_n(x)-w_n(y)|^{p(x,y)}}{|x-y|^{N+s(x,y)p(x,y)}}dxdy
   \leq 2^{p^+} \langle \mathcal{T}(u_n)-\mathcal{T}(u),u_n-u\rangle\to 0~\text{as}~n\to +\infty.
\end{equation}
 Then, thanks to \eqref{s+3} and \eqref{s+4}, we conclude
 $$
 \rho_{s,p}(w_n)\to 0~~\text{as} ~~n\to+\infty.
 $$
 Consequently, $u_n\to u$ in $E_0$.
\end{proof}

 To prove Theorem \ref{theo1.1}, we will use  Krasnoselskii's genus theory.  To this end, let us recall the  notion of genus and its basic properties, which can be found in \cite{Chang, Kra}.\\
Let $E$ be a real Banach space. We set
\begin{equation*}
  \mathcal{A}=\{Z\subset E\backslash\{0\}: Z ~\text{is compact and}~ Z=-Z\}.
\end{equation*}
Let $Z\in \mathcal{A} $ and $E=\mathbb{R}^k$. The genus $\gamma(Z)$ of $Z$ is defined by
\begin{equation*}
  \gamma(Z)=\min\{k\geq1: ~\text{there exists an odd continuous mapping~} h: Z\to\mathbb{R}^k\backslash\{0\} \}.
\end{equation*}
Moreover, if such function  does not exist then $\gamma(Z)=\infty$ and by convenience $\gamma(\emptyset)=0$. As a typical example of a set of genus $k$, we can mention a set homeomorphic to a $(k-1)$-dimensional sphere via an odd map.\\
\begin{lem}
Let $E=\mathbb{R}^k$ and $\partial\Omega$ be the boundary of an open, symmetric and bounded subset
$\Omega\subset \mathbb{R}^N$ with $0\in \Omega$. Then $\gamma(\partial\Omega) = N$.
\end{lem}
Moreover, in order to prove Theorem \ref{theo1.1}, we use the following theorem due to Clarke \cite{DC}.
\begin{theorem}[\cite{DC}]\label{Clarke's theorem}
Let $\mathcal{H}\in C^{1}(E,\mathbb{R})$ be a functional satisfying the following conditions
\begin{itemize}
\item $(i)$ The functional $\mathcal{H}$ satisfies the $(PS)$ condition;
\item $(ii)$ $\mathcal{H}$ is bounded from below and even;
\item $(iii)$ there is a compact set $Z\in \mathcal{A}$ such that $\gamma(Z)=j$ and $\sup_{x\in Z}\mathcal{H}(x)<\mathcal{H}(0).$
\end{itemize}
Then $\mathcal{H}$ possesses at least $j$ pairs of distinct critical points, and their corresponding critical values  are less than $\mathcal{H}(0)$.
\end{theorem}

In the light of the variational structure of \eqref{main}, we look for critical points of the associated Euler-Lagrange functional
$\mathcal{H}:E_0\to \mathbb{R}$ defined as
\begin{equation}\label{e3.1}
\mathcal{H}(\varphi) =\widehat{ \mathcal{K}}\left(\Lambda_{p,s}(\varphi)\right)-\int_\Omega F(x,\varphi)dx,
\end{equation}
for all $\varphi\in E_0$, where
$$
\Lambda_{p,s}(\varphi)=\int_{\mathcal{Q}}\frac{1}{p(x,y)}\frac{|\varphi(x)-\varphi(y)|^{p(x,y)}}{|x-y|^{N+s(x,y){p(x,y)}}} dx\,dy,~~\widehat{\mathcal{K}}(t)=\int_0^t\mathcal{K}(s)ds.
$$
Note that $\mathcal{H}$ is a $C^1(E_0, \mathbb{R})$ functional and
\begin{equation} \label{deriv-fun}
\langle \mathcal{H}'(\varphi),\psi\rangle = \mathcal{K}\left(\Lambda_{p,s}(\varphi)\right) \int_{\mathcal{Q}} \frac{|\varphi(x)-\varphi(y)|^{p(x,y)-2}(\varphi(x)-\varphi(y))(\psi(x)-\psi(y))}{|x-y|^{N+p(x,y)s(x,y)}}dxdy -\int_\Omega f(x,\varphi)\psi dx,
\end{equation}
for any $\psi\in E_0$. Thus, critical points of $\mathcal{H}$ are weak solutions of \eqref{main}.
\section{\textbf{Main results and proofs}}
Before stating our first result, we make the following assumptions on $f$:
\begin{description}
   \item[($f_1$)] There exist $c_1>0$ and $1<q(x)<p^*_s(x)$ for all $x\in \Omega$ such that
$$
|f(x,\zeta)| \leq c_1(1+|\zeta|^{q(x)-1}),  \mbox{for all } (x,\zeta) \in \Omega\times \mathbb{R};
 $$
   \item[($f_2$)] there exist $c_2>0$, $\alpha_0\in (1, \theta p^-)$ and an open set $\Omega_0\subset \Omega$ such that
\begin{equation*}
  F(x,\zeta)\geq c_2 |\zeta|^{\alpha_0}~~\text{for all }~~(x,\zeta)\in \Omega_0\times\mathbb{R};
\end{equation*}
   \item[($f_3$)] $f(x,-\zeta)=-f(x,\zeta), \mbox{ for all } (x,\zeta)\in \Omega\times \mathbb{R}$.
 \end{description}
\begin{thm}\label{theo1.1}
Suppose that ($A_1$), ($f_1$)-($f_3$) are satisfied.  If $q^+<\theta p^-$,
then  problem \eqref{main} has infinitely many pairs of weak solutions with negative energy.
\end{thm}
\begin{lem}\label{lemPS}
Suppose that $(A_1)$ and $(f_1)$ are satisfied. Then $\mathcal{H}$ is bounded from below and satisfies the (PS) condition.
\end{lem}
\begin{proof}
From $(A_1)$ and $(f_1)$, we have
\begin{eqnarray*}
\mathcal{H}(u)&=& \widehat{\mathcal{K}}(\Lambda_{p,s}(u))-\int_\Omega F(x,u)dx\\
&\geq& \frac{k_1}{\theta}\left(\Lambda_{p,s}(u)\right)^\theta-\frac{c_1}{q^-}\int_\Omega|u|^{q(x)}dx-c_1|\Omega|,
\end{eqnarray*}
for all $u\in E_0$. Hence by Proposition \ref{prp 3.3}, we obtain
\begin{eqnarray*}
\mathcal{H}(u)&\geq& \frac{k_1}{\theta (p^+)^\theta}\min\left\{\|u\|^{\theta p^+}_{E_0},\|u\|^{\theta p^-}_{E_0} \right\}-\frac{c_1}{q^-}\max\left\{\|u\|^{q^+}_{L^{q(x)}(\Omega)},\|u\|^{q^-}_{L^{q(x)}(\Omega)}\right\}-c_1|\Omega| \\
&\geq& \frac{k_1}{\theta (p^+)^\theta}\min\left\{\|u\|^{\theta p^+}_{E_0},\|u\|^{\theta p^-}_{E_0} \right\}-\frac{cc_1}{q^-}\max\left\{\|u\|^{q^+}_{E_0},\|u\|^{q^-}_{E_0}\right\}-c_1|\Omega|.
\end{eqnarray*}
As $q^+<\theta p^-$, $\mathcal{H}$ is bounded from below and coercive.\\
Let $\{v_j\}$ be a $(PS)$ sequence of $\mathcal{H}$ in $E_0$, that is
\begin{eqnarray}\label{er1}
\mathcal{H}(v_j)~~ \text{is bounded in}~~E_0, \quad \mathcal{H}'(v_j) \to 0 \mbox{ in } E_0^*, \quad \text{as} ~~j \to \infty,
\end{eqnarray}
where $E_0^*$ is the dual space of $E_0$.\\
Thus, by \eqref{er1}, we have
\begin{equation*}
  \langle \mathcal{H}'(v_j), v_j-v \rangle\to 0.
\end{equation*}
Hence
\begin{eqnarray*}
  \langle \mathcal{H}'(v_j), v_j-v \rangle &=& \mathcal{K}(\Lambda_{p,s}(v_j))\int_{\mathcal{Q}} \frac{|v_j(x)-v_j(y)|^{p(x,y)-2}(v_j(x)-v_j(y))((v_j(x)-v(x))-(v_j(y)-v(y))}{|x-y|^{N+p(x,y)s(x,y)}}dxdy \\
   &-& \int_\Omega f(x,v)(v_j-v) dx\to 0.
\end{eqnarray*}
From $(f_1)$, Propositions \ref{prop2.1} and \ref{prp 3.3}, we can easily get that
\begin{equation*}
  \int_\Omega f(x,v)(v_j-v) dx \to 0.
\end{equation*}
Therefore, we have
 \begin{equation*}
   \mathcal{K}(\Lambda_{p,s}(v_j))\int_{\mathcal{Q}}\frac{|v_j(x)-v_j(y)|^{p(x,y)-2}(v_j(x)-v_j(y))((v_j(x)-v(x))-(v_j(y)-v(y))}{|x-y|^{N+p(x,y)s(x,y)}}dxdy\to 0.
 \end{equation*}
The coercivity of $\mathcal{H}$ implies that  $\{v_j\}$ is bounded in $E_0$, passing to subsequence, if necessary, we may assume that
\begin{equation*}
  \Lambda_{p,s}(v_j)\rightarrow d_1\geq 0, ~~ \text{as}~j \to +\infty.
\end{equation*}
If $d_1=0$, then $\{v_j\}$ converge strongly to $v=0$ in $E_0$ and the proof is finished.\\
If $d_1>0$, since the function $\mathcal{K}$ is continuous, we have
\begin{equation*}
  \mathcal{K}(\Lambda_{p,s}(v_j))\to \mathcal{K}(d_1)\geq0,~\text{as}~j\to \infty.
\end{equation*}
Then, by $(A_1)$, for $j$ large enough, we obtain
\begin{equation*}
  0<c_3<\mathcal{K}(\Lambda_{p,s}(v_j))<c_4.
\end{equation*}
It follows that
\begin{equation*}
   \int_{\mathcal{Q}}\frac{|v_j(x)-v_j(y)|^{p(x,y)-2}(v_j(x)-v_j(y))((v_j(x)-v(x))-(v_j(y)-v(y))}{|x-y|^{N+p(x,y)s(x,y)}}dxdy\to 0.
 \end{equation*}
 Finally, Proposition \ref{propS} ensures that  $u_n\to u$  in $E_0$.
\end{proof}
 \subsection{Proof of Theorem \ref{theo1.1}}
 We consider
 \begin{equation*}
   \mathcal{A}_j=\left\{Z\subset\mathcal{A}: ~\gamma(Z)\geq j\right\},
 \end{equation*}
 \begin{equation*}
   d_j=\inf\limits_{Z\in \mathcal{A}_j }\sup\limits_{u\in Z}\mathcal{H}(u),~~~~ j=1,2,\ldots.
 \end{equation*}
 Thus, we have
 \begin{equation*}
   -\infty< d_1\leq d_2\leq\ldots\leq d_j\leq d_{j+1}\leq\ldots.
 \end{equation*}
 Now we prove that $d_j<0$ for every $j\in \mathbb{N}$. For each $j$, we take $j$ disjoint open sets $D_i$ such that $\bigcup_{i=1}^{j}D_i\subset \Omega_0$. For $i=1,2,\ldots,j$, let $u_i\in \left(E_0\cap C^{\infty}_{0}(D_i)\right)\backslash\{0\}$ and
 \begin{equation*}
   Z_j=span\left\{u_1,u_2,\ldots u_j\right\},~~ ~~ S^j_{r_j}=\left\{u\in Z_j:~~\|u\|_{E_0}=r_j \right\},
 \end{equation*}
 where $r_j\in(0,1)$. For each $u\in  Z_j$, there exist $\nu_i\in \mathbb{R}$, $i=1,2,\ldots$ such that
 \begin{equation}\label{et1}
   u(x)=\sum_{i=1}^{j}\nu_iu_i(x)~~~~\text{for}~x\in \Omega.
 \end{equation}
 So
 \begin{equation}\label{et2}
   \|u\|_{L^{\alpha_{0}}(\Omega)}=\left(\int_\Omega |u(x)|^{\alpha_0}\right)^{\frac{1}{\alpha_0}}=\left(\sum_{i=1}^{j}|\nu_i|^{\alpha_0}\int_{D_i} |u_i(x)|^{\alpha_{0}}\right)^{\frac{1}{\alpha_{0}}}.
 \end{equation}
 As all norms of a finite dimensional normed space are equivalent, there is a constant $C>0$ such that
 \begin{equation}\label{et3}
   \|u\|_{E_0}\leq C \|u\|_{L^{\alpha_{0}}(\Omega)}~~\text{for all }~~ u\in Z_j.
 \end{equation}
 By \eqref{et1}, \eqref{et2} and \eqref{et3}, we obtain
 \begin{align*}
   \mathcal{H}(tu) &= \widehat{\mathcal{K}}\left(\Lambda_{p,s}(tu)\right)-\int_\Omega F(x,tu)dx \\
    & \leq \frac{k_2}{\theta}\left(\Lambda_{p,s}(tu)\right)^{\theta}-\sum_{i=1}^{j}F(x,t\nu_iu_i(x))dx\\
    & \leq \frac{k_2}{\theta (p^-)^{\theta}}t^{\theta p^-}\|u\|_{E_0}^{\theta p^-}-a_2t^{\alpha_0}\sum_{i=1}^{j} |\nu_i|^{\alpha_0}\int_{D_i}|u_i(x)|^{\alpha_0}dx\\
    & \leq \frac{k_2r_j^{\theta p^-}}{\theta (p^-)^{\theta}}t^{\theta p^-}-a_2t^{\alpha_0}\|u\|_{L^{\alpha_{0}}(\Omega)}^{\alpha_0}\\
    & \leq \frac{k_2r_j^{\theta p^-}}{\theta (p^-)^{\theta}}t^{\theta p^-}-\frac{a_2r_j^{\alpha_0}}{C^{\alpha_0}}t^{\alpha_0},
 \end{align*}
 for all $u\in S_{r_j}^j$ and sufficient small $t>0$. Since $\alpha_0<\theta p^-$, we can find $t_j\in (0,1)$ and $\epsilon_j>0$ such that
 \begin{equation*}
   \mathcal{H}(t_ju)\leq-\epsilon_j<0~~\text{for all} ~~u\in S_{r_j}^j,
 \end{equation*}
 i.e.,
 \begin{equation*}
   \mathcal{H}(u)\leq-\epsilon_j<0~~\text{for all} ~~u\in S_{t_jr_j}^j.
 \end{equation*}
 It is clear that $\gamma(S_{t_jr_j}^j)=j$ and therefore $d_j\leq -\epsilon_j<0$. Finally, by Lemma \ref{lemPS} and the
results presented above, we can apply Theorem \ref{Clarke's theorem} to show that the functional $\mathcal{H}$ admits at least $j$ pairs of distinct
critical points. Moreover, since $j$ is arbitrary, we obtain infinitely many critical points of $\mathcal{H}$.\\
The proof is complete.\\

Next we will  consider problem \eqref{main} in the case:
\begin{equation*}
  f(x,u)=\mu\mathcal{\omega}_1(x)|u|^{\alpha(x)-2}u-\nu\mathcal{\omega}_2(x)|u|^{\beta(x)-2}u,
\end{equation*}
where $\mu, \nu$ are two real parameters, $\alpha,\beta\in C_+(\Omega)$ and $\mathcal{\omega}_1, \mathcal{\omega}_2$ are functions in some generalized Sobolev spaces, precisely, we assume the following hypothesis.
\begin{enumerate}
  \item[$(A_2)$] $1<\alpha(x)<\beta(x)<p^-\leq p^+<\frac{N}{s}<\frac{\theta N}{s}<\min\{m_1(x),m_2(x)\}$ for all $ x\in \overline{\Omega}$,
  where $m_1,m_2\in C(\overline{\Omega})$, $\mathcal{\omega}_1\in L^{\frac{m_1(x)}{\theta}}(\Omega)$ such that $\mathcal{\omega}_1(x)>0$ in $\Omega_0\subset\subset\Omega$ with $|\Omega_0|>0$ and $\mathcal{\omega}_2\in L^{\frac{m_2(x)}{\theta}}(\Omega)$ such that $\mathcal{\omega}_2(x)\geq 0$ in $\Omega$.
   \end{enumerate}
\begin{thm}\label{thm2}
If $(A_1)$, $(A_2)$ are fulfilled,
then for any $\mu>0$ and $\nu>0$,  problem \eqref{main} admits at least one nontrivial solution.
\end{thm}
For the proof of Theorem \ref{thm2}, we will use the minimum principle.\\
Since $\mathcal{H}$ is weakly lower semi-continuous, it suffices to show that $\mathcal{H}$ is coercive.
\begin{lem}\label{coercive}
Let $(A_1)$ and $(A_2)$ hold. Then for any $\mu>0$ and $\nu>0$ the functional $\mathcal{H}$ is coercive on $E_0$.
\end{lem}
\begin{proof}
By conditions $(A_1)$, $(A_2)$ and H\"{o}lder's inequality, we get that
\begin{align*}
  \mathcal{H}(u) &= \widehat{\mathcal{K}}\left(\Lambda _{p,s}(u)\right)-\mu\int_\Omega\frac{\mathcal{\omega}_1(x)}{\alpha(x)}|u|^{\alpha(x)}dx+\nu \int_\Omega\frac{\mathcal{\omega}_2(x)}{\beta(x)}|u|^{\beta(x)}dx\\
  & \geq \frac{k_1}{\theta \left(p^+\right)^{\theta}}\min\left\{\|u\|_{E_0}^{\theta p^+},\|u\|_{E_0}^{\theta p^-}\right\}-\mu\int_\Omega\frac{\mathcal{\omega}_1(x)}{\alpha(x)}|u|^{\alpha(x)}dx\\
   & \geq \frac{k_1}{\theta \left(p^+\right)^{\theta}}\min\left\{\|u\|_{E_0}^{\theta p^+},\|u\|_{E_0}^{\theta p^-}\right\}-\frac{\mu}{\alpha^-}\|\mathcal{\omega}_1\|_{L^{\frac{m_1(x)}{\theta}}(\Omega)}\left| |u|^{\alpha(x)}\right|_{L^{\frac{m_1(x)}{m_1(x)-\theta}}(\Omega)}\\
   & \geq \frac{k_1}{\theta \left(p^+\right)^{\theta}}\min\left\{\|u\|_{E_0}^{\theta p^+},\|u\|_{E_0}^{\theta p^-}\right\}-\frac{\mu}{\alpha^-}
   \|\mathcal{\omega}_1\|_{L^{\frac{m_1(x)}{\theta}}(\Omega)}\max\left\{ \|u\|^{\alpha^-}_{L^{\frac{m_1(x)\alpha(x)}{m_1(x)-\theta}}(\Omega)}, \|u\|^{\alpha^+}_{L^{\frac{m_1(x)\alpha(x)}{m_1(x)-\theta}}(\Omega)} \right\}\\
    & \geq \frac{k_1}{\theta \left(p^+\right)^{\theta}}\min\left\{\|u\|_{E_0}^{\theta p^+},\|u\|_{E_0}^{\theta p^-}\right\}
    -\frac{\mu}{\alpha^-}   \|\mathcal{\omega}_1\|_{L^{\frac{m_1(x)}{\theta}}(\Omega)}
    \max\left\{c^{\alpha^-} \|u\|^{\alpha^-}_{E_0}, c^{\alpha^+}\|u\|^{\alpha^+}_{E_0} \right\}.
\end{align*}
Since $\alpha^+<p^-<\theta p^-$, we infer that $\mathcal{H}(u)\to +\infty$ as $\|u\|_{E_0}\to +\infty$, that is $\mathcal{H}$ is coercive on $E_0$.
\end{proof}
From Lemma \ref{coercive} and the minimum principle, for any $\mu > 0$ and $\nu > 0$, the functional
$\mathcal{H}$ has a critical point and problem \eqref{main} has a weak solution. The following lemma shows
that it is not trivial.
\begin{lem}\label{lemm4.2}
Suppose that $(A_2)$ holds. Then for any $\mu>0$ and
$\nu>0$, there exists $u_*\in E_0$ such that $u_*\geq0$, $u_*\neq0$ and $\mathcal{H}(tu_*)<0$ for all $t > 0$ small
enough.
\end{lem}
\begin{proof}
 In the sequel, we use the notation $\alpha_{0}^-:=\inf\limits_{x\in\overline{\Omega_0}}\alpha(x)$, $\beta_{0}^-:=\inf\limits_{x\in\overline{\Omega_0}}\beta (x)$ and $p_0^-:=\inf\limits_{x,y\in\overline{\Omega_0}\times \overline{\Omega_0}}p(x,y)$. Since $\alpha_{0}^-<\beta_{0}^-$, let $\epsilon_0>0$ be such that $\alpha_{0}^-+\epsilon_0<\beta_{0}^-$.\\
 Since $\alpha\in C(\overline{\Omega_0})$, there exists $\Omega_1\subset\subset \Omega_0 $ a neighborhood of $x$
 such that $|\alpha(x)-\alpha_{0}^-|<\epsilon_0$ for all $x\in \Omega_1$. So,  $\alpha(x)\leq \alpha_{0}^-+\epsilon_0<\beta_{0}^-$ for all $x\in \Omega_1$.\\
 Take $u_* \in C^\infty_0(\Omega)$ be such that $\overline{\Omega_1}\subset supp(u_*)$, $u_*=1$ for $x\in \overline{\Omega_1}$ and  $0\leq u_*\leq 1$ in $\Omega_0$.\\
Thus, for any $t \in(0, 1)$ we have
\begin{align*}
  \mathcal{H}(tu_*) &= \widehat{\mathcal{K}}\left(\Lambda _{p,s}(tu_*)\right)-\mu\int_\Omega\frac{t^{\alpha(x)}\mathcal{\omega}_1(x)}{\alpha(x)}|u_*|^{\alpha(x)}dx+\nu \int_\Omega\frac{t^{\beta(x)}\mathcal{\omega}_2(x)}{\beta(x)}|u_*|^{\beta(x)}dx\\
  &\leq \frac{k_2t^{\theta p_0^-}}{\theta \left(p_0^-\right)^{\theta}}\left(\rho_{s,p}(u_*)\right)^\theta
  -\frac{\mu t^{\alpha_{0}^-+\epsilon_0}}{\alpha_{0}^-+\epsilon_0} \int_{\Omega_1}\mathcal{\omega}_1(x)| u_*|^{\alpha(x)}\,dx
   +\frac{\nu t^{\beta_{0}^-}}{\beta_{0}^-} \int_{\Omega_0}\mathcal{\omega}_2(x)| u_*|^{\beta(x)}\,dx\\
   &\leq t^{\beta_0^-} \frac{k_2}{\theta \left(p_0^-\right)^{\theta}}\left(\rho_{s,p}(u_*)\right)^\theta
  -\frac{\mu t^{\alpha_{0}^-+\epsilon_0}}{\alpha_{0}^-+\epsilon_0} \int_{\Omega_1}\mathcal{\omega}_1(x)| u_*|^{\alpha(x)}\,dx
   +\frac{\nu t^{\beta_{0}^-}}{\beta_{0}^-} \int_{\Omega_0}\mathcal{\omega}_2(x)| u_*|^{\beta(x)}\,dx.
   \end{align*}
   It follows that $\mathcal{H}(tu_*)<0$ for all $0 < t<\delta^{\frac{1}{\beta_{0}^--\alpha_{0}^--\epsilon_0}}$ with $0 <\delta<\min\{1,\delta_0\}$ and
   $$
   \delta_0:=\frac{\frac{\mu }{\alpha_{0}^-+\epsilon_0} \int_{\Omega_1}\mathcal{\omega}_1(x)| u_*|^{\alpha(x)}\,dx}{
   \frac{k_2}{\theta \left(p_0^-\right)^{\theta}}\left(\rho_{s,p}(u_*)\right)^\theta+\frac{\nu }{\beta_{0}^-} \int_{\Omega_0}\mathcal{\omega}_2(x)| u_*|^{\beta(x)}\,dx}.
   $$
   Finally, we point out that
   $$
   \frac{k_2}{\theta \left(p_0^-\right)^{\theta}}\left(\rho_{s,p}(u_*)\right)^\theta+\frac{\nu }{\beta_{0}^-} \int_{\Omega_0}\mathcal{\omega}_2(x)| u_*|^{\beta(x)}\,dx>0.
   $$
   In fact, if it is not true then
   $$
   \rho_{s,p}(u_*)=0,
   $$
   which gives $\|u_*\|_{E_0} = 0$, hence $u_* = 0$ in $\Omega_0$. This is a contradiction.
\end{proof}
The proof of Theorem \ref{thm3} is now complete.
\begin{thm}\label{thm3}
If $(A_1)$, $(A_2)$ are fulfilled, then there exists $\mu^* > 0$
such that for all $\mu\in (0,\mu^*  )$ and all $\nu > 0$, problem \eqref{main} has at least one non-negative weak
solution.
\end{thm}
\begin{proof}
In this section, we aim to prove Theorem \ref{thm3} by using Ekeland's Variational Principle.
To this aim, we need the following lemma.
\begin{lem}\label{lemm4.1}
There exists $\mu^* > 0$ such that for any
$\mu\in (0,\mu^*)$, $\nu>0$ there exist $\tau, b > 0$
such that $\mathcal{H}(v) \geq b> 0$ for any
$v\in E_0$ with $\|v\|_{E_0}=\tau$.
\end{lem}
\begin{proof}
By Proposition \ref{prp 3.3}, $E_0$ is continuously embedded in $L^{\alpha(x)}(\Omega)$, then there exists   a positive
constant C such that
\begin{equation}\label{geo1}
  \|v\|_{E_0}\leq C \|v\|_{L^{\alpha(x)}(\Omega)} ~~\text{for all}~ v\in E_0.
\end{equation}
Let us assume that $\|v\|_{E_0} <\min\left\{1,\frac{1}{C}\right\}$, where $C$ is given by \eqref{geo1}.  Using the H\"{o}lder's inequality and relation  \eqref{geo1},  we deduce that for any $v \in E_0$ with $\|v\|_{E_0} = \tau\in (0, 1)$ the following
inequalities hold true
\begin{align*}
\mathcal{H}(v) &= \widehat{\mathcal{K}}\left(\Lambda _{p,s}(v)\right)-\mu\int_\Omega\frac{\mathcal{\omega}_1(x)}{\alpha(x)}|v|^{\alpha(x)}dx+\nu \int_\Omega\frac{\mathcal{\omega}_2(x)}{\beta(x)}|v|^{\beta(x)}dx\\
&\geq \frac{k_1}{\theta \left(p^+\right)^{\theta}}\left(\rho_{s,p}(v)\right)^{\theta}
 -\frac{\mu}{\alpha^{-}}\int_{\Omega}\mathcal{\omega}_1(x)|v|^{\alpha(x)}dx\\
 &\geq \frac{k_1}{\theta \left(p^+\right)^{\theta}}\|v\|^{\theta p^+}_{E_0}
 - \frac{\mu}{\alpha^-}C^{\alpha^-}\|\mathcal{\omega}_1\|_{L^{\frac{m_1(x)}{\theta}}(\Omega)}\|v\|^{\alpha^-}_{E_0}\\
&= \frac{k_1}{\theta \left(p^+\right)^{\theta}}\tau^{\theta p^+}
 - \frac{\mu}{\alpha^-}C^{\alpha^-}\tau^{\alpha^-}\|\mathcal{\omega}_1\|_{L^{\frac{m_1(x)}{\theta}}(\Omega)}\\
&= \tau^{\alpha^-}\left( \frac{k_1}{\theta \left(p^+\right)^{\theta}}\tau^{\theta p^+-\alpha^-}
 - \frac{\mu}{\alpha^-}C^{\alpha^-}\|\mathcal{\omega}_1\|_{L^{\frac{m_1(x)}{\theta}}(\Omega)}\right).
 \end{align*}
Putting
\begin{equation}\label{e4.6}
\mu^* = \frac{k_1\alpha^-}{2\theta C^{\alpha^-}(p^+)^\theta\|\mathcal{\omega}_1\|_{L^{\frac{m_1(x)}{\theta}}(\Omega)}}\tau^{\theta p^+ - \alpha^-}.
\end{equation}
Consequently, for all $\mu\in(0,\mu^*)$ and $v\in E_0$ with $\|u\|_{E_0}= \tau$, there exists a positive constant
$b = \tau^{\theta p^+}/(2\theta(p^+)^{\theta})$ such that
$$
\mathcal{H}(v) \geq b > 0.
$$
This completes the proof.
\end{proof}
By Lemma~\ref{lemm4.1}, we have
\begin{equation}\label{e4.9}
\inf_{v\in \partial B_{\tau}(0)} \mathcal{H}(v) > 0,
\end{equation}
where $\partial B_{\tau}(0)=\{v\in E_0;~~ \|v\|_{E_0}=\tau\}$.\\
On the other hand, from Lemma \ref{lemm4.2}, there exists
$u_*\in E_0$ such that $\mathcal{H}(tu_*) < 0$ for $t>0$
 small enough.
Using the proof of Lemma \ref{lemm4.1}, it follows that
$$
\mathcal{H}(v) \geq \frac{k_1}{\theta \left(p^+\right)^{\theta}}\|v\|^{\theta p^+}_{E_0}
 - \frac{\mu}{\alpha^-}C^{\alpha^-}\|\mathcal{\omega}_1\|_{L^{\frac{m_1(x)}{\theta}}(\Omega)}\|v\|^{\alpha^-}_{E_0} \quad\text{for } v\in
B_{\tau}(0).
$$
Thus,
$$
-\infty < \overline{c}_{\mu}
:= \inf_{ \overline{B_{\tau}(o)}} \mathcal{H} < 0.
$$
Now let $\varepsilon$ be such that
$
 0 <\varepsilon < \inf_{\partial B_{\tau}(0)} \mathcal{H}
- \inf_{ \overline{B_{\tau}(0)}} \mathcal{H}.
$
 Then, by applying Ekeland's Variational Principle to the functional
$$
\mathcal{H}: \overline{B_{\tau}(0)}\to \mathbb{R},
$$
there exists $v_{\varepsilon} \in \overline{B_{\tau}(0)}$ such
that
\begin{gather*}
\mathcal{H} (v_{\varepsilon})
\leq \inf_{\overline{B_{\tau}(0)}}\mathcal{H} + \varepsilon,
\\
\mathcal{H} (v_{\varepsilon}) <  \mathcal{H}(v)
 + \varepsilon\|v-v_{\varepsilon}\|_{E_0} ~\text{ for } v\neq
v_{\varepsilon}.
\end{gather*}
Since $ \mathcal{H} (v_{\varepsilon})
< \inf_{ \overline{B_{\tau}(0)}} \mathcal{H}
+ \varepsilon < \inf_{ \partial B_{\tau}(0)} \mathcal{H}$,
we deduce that $v_{\varepsilon}\in B_{\tau}(0)$.\\
Now, we define $\mathcal{H}_1:\overline{B_{\tau}(0)}\to \mathbb{R}$
by
$$
\mathcal{H}_1(v) = \mathcal{H}(v) + \varepsilon\|v - v_{\varepsilon}\|_{E_0}.
$$
It is clear that $v_{\varepsilon}$ is an minimum of
$\mathcal{H}_1$. Therefore, for small $t>0$ and $u\in B_1(0)$, we have
$$
\frac{\mathcal{H}_1(v_{\varepsilon} + t u)
 - \mathcal{H}_1(v_{\varepsilon})}{t} \geq 0,
$$
which implies that
$$
\frac{\mathcal{H}(v_{\varepsilon} + t u)
- \mathcal{H}(v_{\varepsilon})}{t}
+ \varepsilon\|u\|_{E_0} \geq 0.
$$
As $t\to 0$, we obtain
$$
\langle \mathcal{H}'(v_{\varepsilon}), u\rangle
+ \varepsilon\|u\|_{E_0} \geq 0 \quad \text{for all }~ u\in B_1(0).
$$
Hence,
$\|\mathcal{H}'(v_{\varepsilon})\|_{E_0'} \leq \varepsilon$.
We deduce that there exists a sequence $(v_j)_j\subset B_{\tau}(0)$
such that
\begin{equation}\label{e4.10}
\mathcal{H}(v_j) \to \overline{c}_{\mu,\nu}<0 \quad\text{and}\quad
\mathcal{H}'(v_j) \to 0.
\end{equation}
It is clear that $(v_j)$ is bounded in $E_0$. By reflexivity of $E_0$, for subsequence still denoted $(v_j)$, we have $v_j\rightharpoonup v$ in $E_0$.\\
 Next, we show the strong convergence of $(v_j)$ in $E_0$.\\
  Claim:\\
 \begin{equation}\label{con1}
   \lim_{j\rightarrow +\infty}\int_{\Omega}\mathcal{\omega}_1(x)|v_j|^{\alpha(x)-2}v_j(v_j-v)dx=0,
 \end{equation}
  and
  \begin{equation}\label{con2}
 \lim_{j\rightarrow +\infty}\int_{\Omega}\mathcal{\omega}_2(x)|v_j|^{\beta(x)-2}v_j(v_j-v)dx=0.
 \end{equation}
 In fact, from the H\"{o}lder's type inequality, we have
 \begin{align*}
    &\int_{\Omega}\mathcal{\omega}_1(x)|v_j|^{\alpha(x)-2}v_j(v_j-v)dx\\
    &\leq \|\mathcal{\omega}_1\|_{L^{\frac{m_1(x)}{\theta}}(\Omega)}\left\||v_j|^{\alpha(x)-2}v_j(v_j-v)\right\|_{L^{\frac{m_1(x)\alpha(x)}{m_1(x)-\theta}}(\Omega)}\\
    &\leq \|\mathcal{\omega}_1\|_{L^{\frac{m_1(x)}{\theta}}(\Omega)} \left\||v_j|^{\alpha(x)-2}v_j\right\|_{L^{\frac{\alpha(x)}{\alpha(x)-1}}(\Omega)}\left\|v_j-v\right\|_
    {L^{\frac{m_1(x)\alpha(x)}{m_1(x)-\theta\alpha(x)}}(\Omega)}\\
    &\leq  \|\mathcal{\omega}_1\|_{L^{\frac{m_1(x)}{\theta}}(\Omega)}\left(1+\|v_j\|^{\alpha^+-1}_{L^{\alpha(x)}(\Omega)}\right)
   \left\|v_n-v\right\|_{L^{\frac{m_1(x)\alpha(x)}{m_1(x)-\theta\alpha(x)}}(\Omega)}.
 \end{align*}
Since $E_0$ is continuously embedded in $L^{\alpha(x)}(\Omega)$ and $(v_j)$ is bounded in $E_0$, so $(v_j)$ is bounded in $L^{\alpha(x)}(\Omega)$. On the other hand, since the embedding $E_0\hookrightarrow L^{\frac{m_1(x)\alpha(x)}{m_1(x)-\theta\alpha(x)}}(\Omega)$ is compact, we deduce $\left\|v_j-v\right\|_{L^{\frac{m_1(x)\alpha(x)}{m_1(x)-\theta\alpha(x)}}(\Omega)}\to 0$ as $j\rightarrow+\infty$.
Similarly, we get
$$
 \lim_{j\rightarrow +\infty}\int_{\Omega}\mathcal{\omega}_2(x)|v_j|^{\beta(x)-2}v_j(v_j-v)dx=0.
 $$
Hence, the proof of Claim  is complete.\\
Moreover, since $\mathcal{H}'(v_j) \to 0$ and $(v_j)$ is bounded in $E_0$, we have
\begin{align*}
    \left|\langle \mathcal{H}'(v_j),v_j-v\rangle\right|
     &\leq \left|\langle \mathcal{H}'(v_j),v_j\rangle\right|+\left|\langle \mathcal{H}'(v_j),v\rangle\right|\\
    &\leq \|\mathcal{H}'(v_j)\|_{E_0'}\|v_j\|_{E_0}+\|\mathcal{H}'(v_j)\|_{E_0'}\|v\|_{E_0},
    \end{align*}
that is,
$$
\lim_{j\rightarrow +\infty}\langle \mathcal{H}'(v_j),v_j-v\rangle=0.
$$
Therefore
\begin{equation}\label{Ec1}
  \begin{split}
    \lim_{j\rightarrow +\infty}& \left(\mathcal{K}(\Lambda_{p,s}(v_j))\int_{\mathcal{Q}} \frac{|v_j(x)-v_j(y)|^{p(x,y)-2}(v_j(x)-v_j(y))((v_j(x)-v(x))-(v_j(y)-v(y))}{|x-y|^{N+p(x,y)s(x,y)}}dxdy \right. \\
  &\left.-\mu \int_{\Omega}\mathcal{\omega}_1(x)|v_j|^{\alpha(x)-2}v_j(v_j-v)dx+\nu  \int_{\Omega}\mathcal{\omega}_2(x)|v_j|^{\beta(x)-2}v_j(v_j-v)dx\right)=0,
  \end{split}
\end{equation}
Combining this with relations \eqref{e4.10}-\eqref{con2} it follows that
 \begin{equation*}
   \mathcal{K}(\Lambda_{p,s}(v_j))\int_{\mathcal{Q}}\frac{|v_j(x)-v_j(y)|^{p(x,y)-2}(v_j(x)-v_j(y))((v_j(x)-v(x))-(v_j(y)-v(y))}{|x-y|^{N+p(x,y)s(x,y)}}dxdy\to 0.
 \end{equation*}
 Since $\{v_j\}$ is bounded in $E_0$, passing to subsequence, if necessary, we may assume that
\begin{equation*}
  \Lambda_{p,s}(v_j)\rightarrow d_1\geq 0, ~~ \text{as}~j \to +\infty.
\end{equation*}
If $d_1=0$, then $\{v_j\}$ converge strongly to $v=0$ in $E_0$ and the proof is finished.\\
If $d_1>0$, since the function $\mathcal{K}$ is continuous, we have
\begin{equation*}
  \mathcal{K}(\Lambda_{p,s}(v_j))\to \mathcal{K}(d_1)\geq0,~\text{as}~j\to \infty.
\end{equation*}
Then, by $(A_1)$, for $j$ large enough, we obtain
\begin{equation*}
  0<c_3<\mathcal{K}(\Lambda_{p,s}(v_j))<c_4.
\end{equation*}
It follows that
\begin{equation*}
   \int_{\mathcal{Q}}\frac{|v_j(x)-v_j(y)|^{p(x,y)-2}(v_j(x)-v_j(y))((v_j(x)-v(x))-(v_j(y)-v(y))}{|x-y|^{N+p(x,y)s(x,y)}}dxdy\to 0.
 \end{equation*}
 According to the fact that $\mathcal{T}$ satisfies condition $(S^+)$,  we conclude that $v_j\to v$ strongly in $E_0$.\\
Since $\mathcal{H}\in C^1(E_0,\mathbb{R})$, we have
\begin{equation}\label{con3}
  \mathcal{H}'(v_j)\to \mathcal{H}(v)~~\text{as}~j\to+\infty
\end{equation}
Relations \eqref{e4.10} and \eqref{con3} show that
$\mathcal{H}'(v)=0$ and thus $v$ is a weak solution for problem \eqref{main}. Moreover, by relation \eqref{e4.10}, it follows that $\mathcal{H}(v) < 0$ and thus, $v$ is a nontrivial weak solution for \eqref{main}.  The proof of Theorem \ref{thm3} is now completed.
\end{proof}

\end{document}